\documentclass{article}

\usepackage{amsfonts}
\usepackage{amsthm}
\usepackage{amsmath}
\usepackage{amssymb}
\usepackage{mathtools} 
\usepackage{authblk} 
\usepackage{xcolor}
\usepackage{soul} 
\usepackage{enumerate}
\usepackage{float} 
\usepackage{hyperref} 

\usepackage[top=1in, bottom=1in, left=1in, right=1in]{geometry}

\newtheorem{theorem}{Theorem}[section]

\newtheorem{corollary}{Corollary}[section]
\newtheorem{proposition}{Proposition}[section]

\theoremstyle{definition}
\newtheorem{definition}{Definition}[section]

\theoremstyle{remark}
\newtheorem{remark}{Remark}[section]
\newtheorem{example}{Example}[section]

\begin{document}


\title{A Topological Proof for a Version of Artin's Induction Theorem}

\date{}


\author[1]{M\"uge Saadeto\u{g}lu\thanks{Corresponding author. Email: \texttt{muge.saadetoglu@emu.edu.tr}}}

\affil[1]{{\small Department of Mathematics, Eastern Mediterranean University, Famagusta, Northern Cyprus, via Mersin-10, Turkey}}


\maketitle

\vspace{-1cm}

\begin{abstract}
We define a Euler characteristic $\chi(X,G)$ for a finite cell complex $X$ with a finite group $G$ acting cellularly on it. Then, each $K_{i}(X)$ (a complex vector space with basis the $i$-cells of $X$) is a representation of $G$, and we define $\chi(X,G)$ to be the alternating sum of the representations $K_{i}(X)$, as elements of the representation ring $R(G)$ of $G$.  By adapting the ordinary proof that the
alternating sum of the dimensions of the chain complexes is equal to the alternating sum of the dimensions of the homology groups,
we prove that there is another definition of $\chi(X,G)$ with the alternating sum of the representations $H_i(X)$, again as elements of the representation ring $R(G)$.
We also show that the character of this virtual representation $\chi(X,G)$, with respect to a given element $g$, is just the ordinary Euler characteristic of the fixed-point set by this element. Finally, we give a topological proof of a version of
Artin's induction theorem. More precisely, we show that, if $G$ is a group with an irreducible representation of dimension greater than 1, then each character of $G$ is a linear combination with rational coefficients of characters induced up from characters of proper subgroups of $G$. \\

\noindent \textbf{Keywords:} irreducible representation; subrepresentation; character; homology; complex projective space \\

\end{abstract}

\section{Introduction}

The origin of the representation theory of finite
groups goes way back to a correspondence between R. Dedekind and F. G. Frobenius that took
place around a century ago~\cite{lam,conrad}. This theory studies abstract algebraic structures by representing their elements as linear transformations of vector spaces. Specifically, a~representation is a way of making an abstract algebraic object more solid by defining its elements by matrices and arithmetic operations on matrices. Linear algebra and matrix theory are well-known and less abstract, so understanding more abstract objects using familiar linear algebra objects can be helpful in deriving properties and simplifying calculations in more abstract settings. Representation theory is a branch of abstract mathematics with many important applications not only in mathematics but in other sciences as well, including chemistry, physics, computing and statistics. To~name a few examples: in~\cite{stursberg}, an~algorithm to construct skew-symmetric matrices is developed by using real irreducible representations of $SO(3)$ (application in Lie theory); in~\cite{dally}, a non-Abelian 4 dimensional unitary representation of the Braid group is used to obtain a totally leakage-free braiding (application in computer science); and in~\cite{chen}, applications of the group representation theory in various branches of physics and quantum chemistry, in~particular nuclear and molecular physics, are discussed. The~motivation for this work was to use representation theory to give a topological proof for an abstract problem.  
We explain the method of research below: for $X$, a finite cell complex, let $K_{*}(X)$ be the cellular chain
complex with complex coefficients where each of the $K_{i}(X)$s is
a complex vector space with the $i$-cells as a basis. Now, suppose a finite group $G$ acts cellularly on $X$. Then, each
$K_{i}(X)$ is a representation of $G$, and we can make a more
accurate Euler characteristic for $X$ with $G$ acting by taking $\chi(X,G) = \sum_{i=0}^n(-1)^i [K_i(X;\mathbb{C})]\in
R_\mathbb{C}(G).$ The dimension of this virtual representation is just the ordinary
Euler characteristic of $X$. We claim there is another definition of $\chi(X,G)$ as follows, $\chi(X,G) = \sum_{i=0}^n(-1)^i [H_i(X;\mathbb{C})]\in
R_\mathbb{C}(G)$, where $H_i(X)$ is the $i$th homology group of $X$ with $G$ acting on
it. In~this paper, we will work on the following problems. We will try to adapt the ordinary proof that
$\sum_{i}(-1)^i \textrm{dim}_\mathbb{C} K_i(X;\mathbb{C})= \sum_{i} (-1)^i
\textrm{dim}_\mathbb{C} H_i(X;\mathbb{C})$ to give a similar formula with equality as elements of
$R_\mathbb{C}(G)$ when $G$ acts on $X$. We will also look for a formula for the character
of $\chi(X,G)$ in terms of the ordinary Euler characteristic of $X^g
= \{ x \in X |\,gx = x \}$. Finally, with~the help of these results, we give a topological proof of a version of Artin's induction theorem. Artin's induction theorem says that any character of a finite group can be expressed as a rational linear combination of characters that are induced from the cyclic subgroups. In~this work, we prove a version of this theorem; more precisely, we show that if~$G$ is a group with an irreducible representation of dimension greater than 1, then each character of $G$ can be expressed as a rational linear combination of characters that are induced up from characters of proper subgroups of $G$.

The structure of the paper is as follows. Section~\ref{background} will give the preliminaries on representation theory and topology. Sections~\ref{sec6.2} and \ref{sec6.3} will give examples for the Euler characteristic computation using definition with the chain groups and the homology groups, respectively. Finally, we give the results in Section~\ref{sec6.4}.

This work is presented in International Congress of Mathematicians, ICM2014, which took place in Seoul, South Korea in~2014.

\section{Preliminaries}\label{background} 

Our main references here are~\cite{armstrong,kao,serre}. {Note that the underlying field is the field of complex numbers.}





\begin{definition}
Let $\varphi:G\rightarrow GL(V)$ be a linear representation, and~let $W$ be a subspace of $V$. Suppose that $W$ is invariant under the action of $G$, that is to say, suppose that $w\in W$ implies that $\varphi_s(w)\in W$ for all $s\in G$. The~restriction $\varphi^W$ of $\varphi_s$ to $W$ is then an isomorphism of $W$ onto itself. Thus, $\varphi^W:G\rightarrow GL(W)$ is a linear representation of $G$ in $W$ and~is said to be a subrepresentation of $G$.

\end{definition}

\begin{definition}
A {non-zero linear representation is said to be irreducible if it has no proper non-trivial subrepresentation.}
\end{definition}

\begin{definition}
Let $G$ be a finite group, let $H$ be a subgroup of index $n$, and~let $(\pi,V)$ be any representation of $H$. Let $x_1, x_2,…,x_n$  be representatives in $G$ of the cosets in $G/H$. The~induced representation ${Ind_H}^G\pi$ acts on $W=\bigoplus{x_iV}$,  where $i$ ranges over the coset representatives, and~via this, $G$ acts on $W$ as follows: $g.\sum x_iv_i=\sum x_{j(i)}\pi(h_i)v_i$ for $v_i\in V$.

\end{definition}

Below, we state some well-known theorems on representation theory that will be useful in the coming~sections.

\begin{theorem}[Fixed-Point Formula]\label{fpf}  Let $V$ be a representation of a finite group $G$, and~let $X$ be a finite $G$-set. Then, the number of left fixed elements (by the action of $g$) in $X$ is $\chi_V(g)$ for every $g\in G$.

\end{theorem}

\begin{theorem}
The number of conjugacy classes of $G$ is the same as the number of irreducible representations of $G$, up~to isomorphism.
\end{theorem}

\begin{theorem}
Group $G$ is Abelian if and only if all irreducible representations have degree $1$.
\end{theorem}

\begin{theorem}
The sum of the squares of the dimensions of distinct irreducible representations is the same as the order of the given group $G$.
\end{theorem}
\newpage
\begin{theorem} \label{teo9} Let $V$ be a linear representation of $G$ with~character $\phi$, and~suppose $V$ decomposes into a direct sum of
irreducible representations $V=W_1\oplus \ldots \oplus W_k$. Then, if~$W$ is an irreducible representation with character $\chi$,
the number of $W_i$ isomorphic to $W$ is equal to the scalar product
$(\phi|\chi)$, where
\begin{displaymath}
       (\phi|\chi)=\frac{1}{|G|}\sum_{t\in G}\phi(t)\chi(t).
\end{displaymath}
\end{theorem}

Next, we state some well-known theorems in algebraic topology. $K$ below is a finite simplicial~complex.

\begin{theorem} $H_0(K)$ is a free Abelian group whose rank is the number of connected components of~$|K|$. \end{theorem}

\begin{theorem} If $|K|$ is connected, Abelianizing its fundamental group gives the first homology group of $K$. \end{theorem}

\begin{theorem} $H_2(K)\cong\mathbb{Z}$ if $|K|$ is an orientable surface and~is $0$ if it is not. \end{theorem}

\begin{remark} {Examples in Sections~\ref{sec6.2} and \ref{sec6.3} below give
evidence for Theorem \ref{thm2}, which will be proved in
Section~\ref{sec6.4}.}
\end{remark}

\section{Evaluating \textbf{\boldmath{$\chi(X,G)$}} Using the Definition with the Chain
Groups} \label{sec6.2}

 In this section, we evaluate $\chi(X,G)$ for some examples using its
definition with the chain~groups.

\begin{example}$D_3$ acting on the equilateral triangle as the group
of symmetries.\end{example}

{$D_3$ here stands for the dihedral group of degree $3$, or~in other words, the dihedral group of order $6$. It is the group of $6$ symmetry transformations of an equilateral triangle. They are reflections in the axes through $3$ vertices and~also rotations around the center by multiples of $120^\circ$. Let us denote the rotations by $e,\rho, \rho^2$
and the reflections by $\alpha, \beta, \gamma$.} The group $D_3$ has $2$ one-dimensional (trivial and alternating),  and~$1$ two-dimensional (standard) irreducible representations; see~\cite{serre,kao}. Let us call them $\textbf{1},\sigma,\tau$, respectively. For~the trivial representation, $\chi_\textbf{1}(g)=1$ for all $g\in G$. Character values of the alternating representation are given by $\chi_\sigma(g)=1$ for  $g\in A_3$ and~$\chi_\sigma(g)=-1$  for $g\not \in A_3$. Finally, the~$2$ dimensional $\tau$ has character values $\chi_\tau(e)=2$, $\chi_\tau(\alpha)=0$ and $\chi_\tau((123))=-1$.

Here, $K_0$ is the representation space of dimension $3$, which is
the number of $0$-cells. Let $P$ be the representation corresponding
to this. By~the fixed-point formula, Theorem \ref{fpf}, the~character values of this representation are shown in Table~\ref{table2}.

\begin{table}[H]\caption{Character table for the representation $P$.}\label{table2} 

 \setlength{\tabcolsep}{8.65mm}\begin{tabular}{c|cccccc}
\noalign{\hrule height 1.0pt}
 & $e$ & $\rho$ & $\rho^2$ & $\alpha$ & $\beta$ & $\gamma$ \\ \noalign{\hrule height 0.5pt}
   $\chi_P$ & 3 & 0 & 0 & 1 & 1 & 1 \\\noalign{\hrule height 1.0pt}
 \end{tabular}
\end{table}


By Theorem \ref{teo9}, the~number of irreducibles isomorphic to $\textbf{1}$ is $(\chi,
\phi_1)=1$, where $\phi_1$ stands for the character of the trivial representation. Similarly, $(\chi,
\phi_2)=0$ and $(\chi,\phi_3)=1$, where $\phi_2$ and $\phi_3$ are the
characters of $\sigma$ and $\tau$, respectively. Therefore, $P=\textbf{1}\oplus \tau.$

Similarly, $K_1$ is the representation space of dimension $3$.
Let us call this representation $Q$. The~character values for $Q$ are presented in Table~\ref{table3}.

\begin{table}[H]\caption{Character table for the representation $Q$.}\label{table3}

 \setlength{\tabcolsep}{8.03mm} \begin{tabular}{c|cccccc}
 \noalign{\hrule height 1.0pt}
 & $e$ & $\rho$ & $\rho^2$ & $\alpha$ & $\beta$ & $\gamma$ \\ \noalign{\hrule height 0.5pt}
   $\chi_Q$ & 3 & 0 & 0 & $-$1 & $-$1 & $-$1 \\ \noalign{\hrule height 1.0pt}
 \end{tabular}
\end{table}

The scalar products are given by $(\chi,\phi_1) = 0$, $(\chi,\phi_2) = 1$ and
$(\chi,\phi_3)= 1$. Hence, $Q=\sigma\oplus \tau.$

Finally, $K_2$ is a $1$-dimensional representation space, which we shall call $R$. The~character values of $R$ are given in Table~\ref{table4}.

\begin{table}[H]\caption{Character table for the representation $R$.}\label{table4}
\centering
  \setlength{\tabcolsep}{8.03mm}\begin{tabular}{c|cccccc}
  \noalign{\hrule height 1.0pt}
 & $e$ & $\rho$ & $\rho^2$ & $\alpha$ & $\beta$ & $\gamma$ \\  \noalign{\hrule height 0.5pt}
   $\chi_R$ & 1 & 1 & 1 & $-$1 & $-$1 &$ -$1 \\  \noalign{\hrule height 1.0pt}
 \end{tabular}
\end{table}

Here, the~scalar products have values $(\chi,\phi_1)=0$, $(\chi,\phi_2)=1$ and
$(\chi,\phi_3)=0$. Therefore $R=\sigma.$

Hence,
$$\begin{array}{cl}
 \chi(X;G)  & =P\oplus-Q \oplus R=(\textbf{1}\oplus
\tau) -(\rho\oplus \tau)+ (\rho) \\
   & =\textbf{1}.
\end{array}$$

\begin{example}$C_2 \times C_2$ acting on an octahedron as rotations through
$\pi$ degrees about the three axes through opposite vertices.\end{example}

{$C_2$ here represents the cyclic group of order $2$. Thus the Abelian group $C_2 \times C_2$ has four elements, identity element of order $1$, and~the remaining three elements of order $2$. Let us represent these elements by}
$e,\rho_1,\rho_2,\rho_3$, where $\rho_1$ is the axis joining vertices $5$ and $6$, $\rho_2$ is the axis joining vertices $3$ and $4$, and finally $\rho_3$ is the axis joining vertices $1$ and $2$; see Figure~\ref{figocta1}.
\begin{figure}[H]
\includegraphics[width=0.5\textwidth]{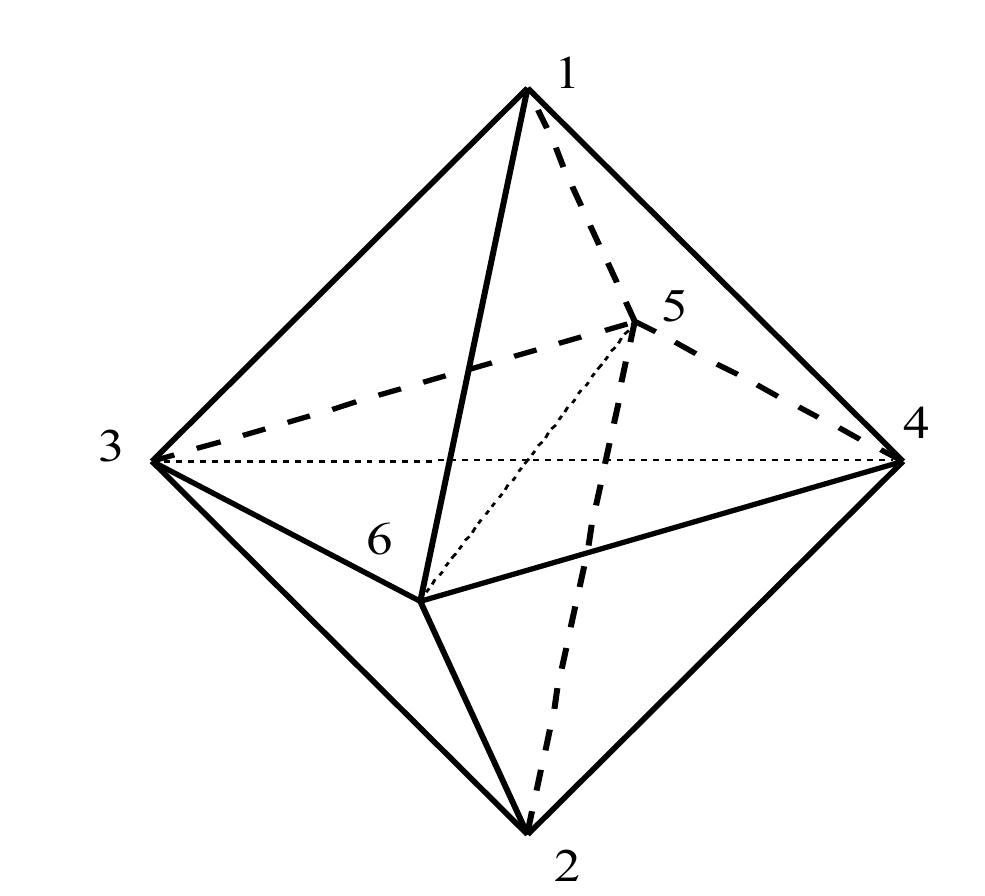}
\caption{$C_2 \times C_2$ acting on the octahedron as rotations by $180^\circ$ about the three axes joining opposite~vertices.} \label{figocta1}
\end{figure}


As this group is Abelian,  all the irreducible representations have a degree of 1. Furthermore,  as~the sum of the squares of dimensions of distinct irreducible representations is on the order of $C_2\times C_2$, there must be $4$ of them. By~\cite{serre,kao}, these are shown in Table~\ref{table1}.


\begin{table}[H]\caption{Character table for $C_2\times C_2$.}\label{table1}

 \setlength{\tabcolsep}{11.7mm}\begin{tabular}{c|cccc}
 \noalign{\hrule height 1.0pt}
 & e & $\rho_1$ & $\rho_2$ & $\rho_3$ \\ \noalign{\hrule height 0.5pt}
   $\chi_\textbf{1}$ & 1 & \,1 & \,1 &\, 1 \\
   $\chi_{T_1}$      & 1 & \,1 & $-$1  & $-$1 \\
   $\chi_{T_2}$      & 1 &$-$1  &\, 1 & $-$1 \\
   $\chi_{T_3}$      & 1 & $-$1  & $-$1  &\, 1 \\ \noalign{\hrule height 1.0pt}
 \end{tabular}
\end{table}


Let $P$ be the representation corresponding to the space $K_0$ of dimension 6. Similar work to the previous example shows that
$P= 3\textbf{1}\oplus T_1\oplus T_2\oplus T_3.$ Next, let
$Q$ be the representation corresponding to the space $K_1$ of~dimension 12. We obtain
$Q=3\textbf{1}\oplus3T_1\oplus3T_2\oplus3T_3.$ Finally, if~$R$ is the representation corresponding to the space $K_2$, of~dimension 8, then
$R=2\textbf{1}\oplus2T_1\oplus2T_2\oplus2T_3.$ Thus, we obtain
\begin{eqnarray*}\chi(X;G)&=&3\textbf{1} \oplus T_1 \oplus T_2 \oplus T_3-
(3\textbf{1} \oplus 3T_1 \oplus 3T_2 \oplus 3T_3)+ (2\textbf{1}
\oplus 2T_1 \oplus 2T_2 \oplus 2T_3)\\
&=&\textbf{1} \oplus \textbf{1}.\end{eqnarray*}

\begin{example}$C_2\times C_2\times C_2$ acting on the octahedron as reflections in the three coordinate planes.\end{example}

{Following from the previous example, $C_2\times C_2\times C_2$ is an Abelian group of order $8$ with~an identity element of order $1$ and~the remaining seven elements of order $2$.} Call the $\alpha$ reflection in the $xy$-plane, $\beta$ reflection in the $yz$-plane, and~$\gamma$ reflection in the $xz$-plane; see Figure~\ref{figocta2}. Hence, the~elements
of $C_2\times C_2\times C_2$ are $e$, $\alpha$, $\beta$, $\gamma$,
$\alpha\circ\beta$, $\alpha\circ\gamma$, $\beta\circ\gamma$,
$\alpha\circ\beta\circ\gamma$ (we will use the notation $\alpha\beta$ for~the element $\alpha\circ\beta$).
Similar calculations to previous examples show that $ \chi(X,G)=V_1\oplus V_8$, where $V_1$ is the trivial representation, and~$V_8$ is the one mapping
$e$, $\alpha\beta$, $\alpha\gamma$, $\beta\gamma$ to $1$ and~the rest of the elements to $-1$.
(The group $C_2\times C_2\times C_2$ is Abelian; therefore, it has $8$ irreducible representations of degree $1$; see Table~\ref{tablenew}  \cite{serre,kao}).

\vspace{-9pt}
\begin{figure}[H]
\includegraphics[width=0.5\textwidth]{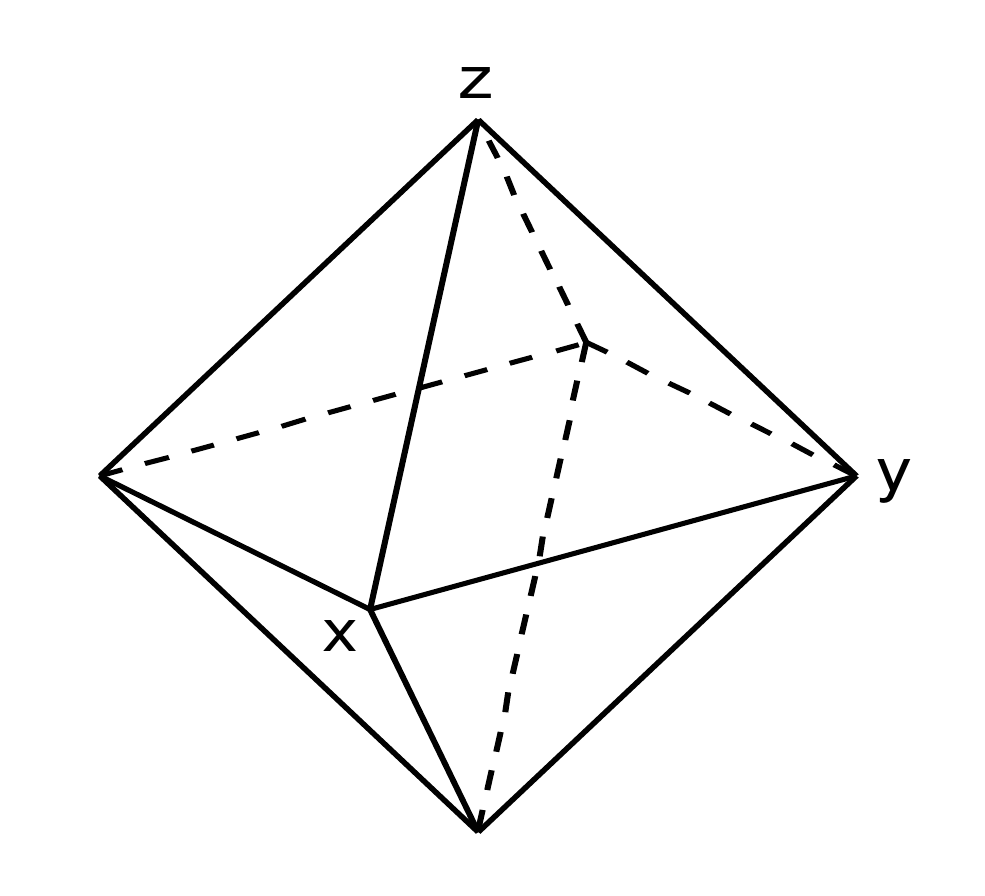}
\caption{$C_2\times C_2\times C_2$ acting on the octahedron. 
}\label{figocta2}
\end{figure}

\begin{table}[H]\caption{Character table for $C_2\times C_2\times C_2$.}\label{tablenew}

 \setlength{\tabcolsep}{5.4mm}\begin{tabular}{c|cccccccc}
 \noalign{\hrule height 1.0pt}
 & e & $\alpha$ & $\beta$ & $\gamma$ & $\alpha\beta$ & $\alpha\gamma$ & $\beta\gamma$ & $\alpha\beta\gamma$  \\ \noalign{\hrule height 0.5pt}
   $\chi_{V_1}$ & 1 & \,1 & \,1 & \,1 & \,1 & \,1 &\, 1 & \,1\\
   $\chi_{V_2}$ & 1 & \,1 & \,1 & $-$1  & \,1 &$-$1  &$-$1  & $-$1 \\
   $\chi_{V_3}$ & 1 & \,1 & $-$1  & \,1 & $-$1  & \,1 & $-$1  & $-$1 \\
   $\chi_{V_4}$ & 1 & \,1 & $-$1  & $-$1  &$ -$1  & $-$1  & \,1 & \,1 \\
   $\chi_{V_5}$ & 1 & $-$1  & \,1 & \,1 & $-$1  & $-$1  & \,1 &$ -$1\\
   $\chi_{V_6}$ & 1 & $-$1  & \,1 & $-$1  &$-$1  & \,1 & $-$1  & \,1 \\
   $\chi_{V_7}$ & 1 & $-$1  & $-$1  & \,1 & \,1 & $-$1  & $-$1  & \,1 \\
   $\chi_{V_8}$ & 1 & $-$1  & $-$1  & $-$1  & \,1 & \,1 & \,1 & $-$1 \\ \noalign{\hrule height 1.0pt}
 \end{tabular}
\end{table}
\unskip

\section{Evaluating the \textbf{\boldmath{$\chi(X,G)$}} Using the Definition with the Homology
Groups} \label{sec6.3}

In this section,  we will evaluate $\chi(X,G)$ for the examples above, this time by using its definition with the homology~groups.


\begin{example}
The group $D_3$ acting on the equilateral triangle as the
group of symmetries.\end{example} Irreducible representations are $\textbf{1}$,
$\sigma$, $\tau$. We use the~formulae

$$\chi(X,G) = \sum_{i=0}^n(-1)^i [H_i(X)].$$

No transformation affects connectedness; thus group $H_0\cong\mathbb{C}$, with~dimension $1$,
contributes $\textbf{1}$ to $[H_0]$. The~group $H_1=0$ with dimension $0$; therefore, there is no contribution to $[H_1]$.
Finally, the second homology group is also $0$; thus, there is no contribution to
$[H_2]$ either. Hence, $\chi(X,G)=\textbf{1}.$

\begin{example}The group $C_2 \times C_2$ acting on the octahedron as the
group of rotations. \end{example} Irreducible representations are
$\textbf{1},T_1,T_2,T_3$.

For the group $H_0\cong\mathbb{C}$ of~dimension $1$, as in the previous example, the~contribution to $[H_0]$ is $\textbf{1}$. The~first homology group is $0$, so there is no contribution to $[H_1]$. Finally, $H_2\cong\mathbb{C}$,
generated by the sum of all the triangles in the triangulation of
the octahedron. If~we call this representation $P$, again, by~the fixed-point formula, the~character values
for $P$ are given by $\chi_P(e)=1, \chi_P(\rho_1)=1, \chi_P(\rho_2)=1, \chi_P(\rho_3)=1$, so $H_2$ contributes $\textbf{1}$ to $[H_2]$. Hence,
$\chi(X,G)=\sum_{i=0}^n(-1)^i [H_i(X;\mathbb{C})]=\textbf{1}\oplus
\textbf{1}.$

\begin{example}The group $C_2\times C_2\times C_2$ acting as
reflections of the octahedron in the three coordinate planes.\end{example}
The group $H_0$ is isomorphic to $\mathbb{C}$ contributing
$\textbf{1}$ to $[H_0]$. $H_1\cong0$, so there is no contribution to
$[H_1]$, and~finally, $H_2\cong\mathbb{C}$ with dimension $1$. If~$P$
is the representation corresponding to this, then $P$ has character values $1$ and $-1$ for the rotations and the reflections, respectively.
As stated earlier, this representation is $V_8$. Hence,
$\chi(X;G)=\textbf{1}\oplus V_8.$

\section{Results}\label{sec6.4}

\begin{theorem}\label{thm2} As elements of the representation ring,
$$\sum_{i=0}^n(-1)^i [K_i(X;\mathbb{C})] = \sum_{i=0}^n(-1)^i
[H_i(X;\mathbb{C})]$$
 for a general case.
 \end{theorem}

\begin{proof}

Let $Z_n$ be the group of $n$-cycles and $B_n$ be the group of bounding $n$-cycles. There are
two short exact~sequences
$$0\rightarrow Z_{n}\rightarrowtail K_{n}\twoheadrightarrow
B_{n-1}\rightarrow 0$$
$$0\rightarrow B_{n}\rightarrowtail Z_{n}\twoheadrightarrow
H_{n}\rightarrow 0$$
and these sequences split. Thus, we have $K_{n}=Z_{n}\oplus B_{n-1}$
and $Z_{n}=B_n\oplus H_{n}$. This implies that $K_{n}=B_n\oplus
H_{n}\oplus B_{n-1}$. Therefore, $\chi_{K_{n}}=\chi_{B_{n}}+\chi_{H_{n}}+\chi_{B_{n-1}}.$

 By the above equation, we~have

$$\sum_{i=0}^n(-1)^i\chi_{K_{i}}=\sum_{i=0}^n(-1)^i(\chi_{B_{i}}+\chi_{H_{i}}+\chi_{B_{i-1}}).$$

 All $\chi_{B_{i}}$s and $\chi_{B_{i-1}}$s cancel except for
$\chi_{B_{-1}}$ and $\chi_{B_{n}}$, so let us consider these cases
now. We have $\chi_{B_{-1}}=0$ \ since
\ $\chi_{K_0}=\chi_{B_0}+\chi_{H_0}$. Additionally, $B_n=0$, as there are no
\{$n+1$\}-dimensional faces. Hence, $\chi_{B_n}=0$.~Therefore,
$$\chi(X;G):=\sum(-1)^i\chi_{K_i}=\sum(-1)^i\chi_{H_i}$$
which implies that
$$\sum_{i=0}^n(-1)^i
[K_i(X;\mathbb{C})]=\sum_{i=0}^n(-1)^i [H_i(X;\mathbb{C})].$$ \end{proof}

\begin{proposition} The character of $\chi(X;G)$ is equal to the ordinary Euler characteristic of $X^g$ where $X^g = \{ x \in X |\,gx = x
\}$. \end{proposition}


\begin{proof}

Let $\nu$ denote the character of $\chi(X;G)$ and let $\chi$ denote the
ordinary Euler characteristic. We~have
$$\nu(\chi(X;G)(g))=\sum(-1)^i\nu_{K_i}(g).$$

However, $\nu_{K_i}(g)$ is the number of $i$-cells in $X$ fixed by $g$,
and hence, $\sum(-1)^i\nu_{K_i}(g)$ is equal to~$\chi(X^g)$. \end{proof}

\begin{remark}\label{remark}If a finite $G$ acts on any set $S$ and $\sigma\in S$ with
Stab$_G(\sigma)=H$, then the $G$-set \, $G\cdot\sigma$ and $G/H$ are
isomorphic (the proof of the orbit-stabilizer theorem). Therefore,
$\mathbb{C}\{G/H\}$ is isomorphic to the vector space with basis
$G.\sigma$. For~a $G$-$CW$-complex, if~$\sigma$ is an $n$-cell whose
stabilizer is $H$, then the orbit $G.\sigma$ contributes one copy of
$\mathbb{C}\{G/H\}$ to $K_n(X)$. The~induced representation of any subgroup's trivial representation is the permutation representation on its cosets, so
$$\mathbb{C}\{G/H\}\cong\textrm{Ind}_H^G(1_H),$$ and
$\chi_{Ind_H^G(\textbf{1}_H )\,\, (g)}$ is equal to the
number of elements in the orbit $G.\sigma$ of $\sigma$ fixed by $g$.
\end{remark}

\begin{theorem}[Artin's Induction Theorem (weaker version)] Let $G$ be
a group with an irreducible representation of dimension greater than
$1$. Then each character of $G$ is a linear combination with
rational coefficients of characters induced up from characters of
proper subgroups of $G$.
\end{theorem}

\begin{proof}

It suffices to show that the trivial representation \textbf{1} can
be written this way, as~by Frobenius reciprocity~\cite{serre},
$$\textrm{Ind}_H^G(a \otimes \textrm{Res}_H^G(b))=\textrm{Ind}_H^G(a)\otimes b.$$

 Therefore, if
\begin{displaymath}
\textbf{1}=\sum_{H_i,\varphi_i,\lambda_i}\lambda_i
Ind_{H_i}^G(\varphi_i)
\end{displaymath} then
\begin{displaymath}
b=\textbf{1}\otimes b =
\sum_{H_i,\varphi_i,\lambda_i}\lambda_iInd_{H_i}^G(\varphi_i \otimes
Res(b)).
\end{displaymath}

To prove the theorem, we need a space $X$ with the following three properties: Euler characteristic of $X$ must be nonzero, $G$ should act on $X$ with all cell stabilizers' proper subgroups, and~the action of $G$ on
$H_*(X)$ must be trivial. $\mathbb{C}P^n$, the~complex projective $n$-space, is in fact the space we need.
Given $b:G\rightarrow GL(V)$, an~$(n+1)$-dimensional complex
representation of $G$, we get an action of $G$ on $\mathbb{C}P^n$
as $g[\underline{v}]=[b(g)\underline{v}]$. Here we use a triangulation of $CP^n$ where the simplices are
preserved by the G-action, \cite{illman}.
Referring to~\cite{munkres} or~\cite{hatcher}, $\mathbb{C}P^n$ is a $CW$ complex of dimension $2n$. In~fact, this space has one open cell in each even dimension $2j$. Therefore,

\begin{displaymath}
 H_*(\mathbb{C}P^n)= \left\{ \begin{array}{ll}
\mathbb{Z} & \textrm{if} \,\, * = 0,2,4,\ldots,2n\\ 0 &
\textrm{otherwise}.
\end{array} \right.
\end{displaymath}

Here, $\chi(\mathbb{C}P^n)=\sum_{q=0}^n(-1)^q\beta_q$,
where $\beta_q$ is the rank of the free
Abelian part of $H_q(\mathbb{C}P^n)$, so $\chi(\mathbb{C}P^n)= n+1$,
and hence is~nonzero.

It is easy to see that Stab$_G([\underline{v}])=G$ if and only if
$\{\lambda\underline{v}\mid
\lambda\in\mathbb{C}\}=\langle\underline{v}\rangle$ is a
sub-representation of $G$. Thus, if~$V$ has no $1$-dimensional
summands, then each Stab$_G([\underline{v}])$ is a proper subgroup. Thus, $\mathbb{C}P^n$ satisfies all the three properties we have
listed above. Then Theorem \ref{thm2} gives
\begin{displaymath}
\chi(X)\cdot \textbf{1}=\sum_i(-1)^iK_i(X)
\end{displaymath}
and
\begin{displaymath}
K_i(X)\cong\oplus_j\mathbb{C}[G/{\operatorname{Stab}}_G(\sigma_j)]
\end{displaymath}
where $j$ ranges over the orbit representatives $\sigma_j$ of
$i$-cells. Hence, the theorem is proved by Remark \ref{remark}. \end{proof} 

\begin{corollary} If $G$ is non-Abelian, then it has an irreducible representation
of dimension greater than $1$, and~hence satisfies the theorem
above. \end{corollary}

\section{Conclusions}
In this paper, given a finite cell complex $X$ with a finite group $G$ acting cellularly on it, we define a more accurate Euler characteristic for the finite cell complex with this action as alternate sums of the chain complexes and the homology groups considered as elements of the representation ring. We prove that both of the definitions are equivalent and that the character of this virtual representation with respect to a given element $g$ is just the ordinary Euler characteristic of the fixed point set by this element. Finally we use representation theory to give a topological proof of a weaker version of
Artin's induction theorem. E.~Artin's induction theorem says that any character of a finite group can be expressed as a rational linear combination of characters that are induced from the cyclic subgroups. In~our work we prove a version of this theorem; in~fact, we show that, if~$G$ is a group with an irreducible representation of dimension greater than 1, then each character of $G$ is a rational linear combination of characters induced up from characters of proper subgroups of $G$.\\


\subsection*{Competing interests}

The author declares no conflict of~interest.

\subsection*{Funding}

This research received no external funding.

\subsection*{Acknowledgements}

The author would like to thank Prof. Ian Leary for his ideas and valuable discussions through the development of this work.

\end{document}